\documentclass[11pt]{article}
\usepackage{ulem}
\usepackage{graphicx,amsmath,amsfonts,amssymb,color,bbm}
\usepackage{epsfig}

 \newcommand{\pend}{\hfill \thicklines \framebox(5.5,5.5)[l]{}}
 \newenvironment{proof}{\noindent {\sc  Proof.} \rm}{\pend}

\numberwithin{equation}{section}

 \newtheorem{lemma}{Lemma}[section]
 
 \newtheorem{proposition}{Proposition}[section]
 
 \newtheorem{remark}{Remark}[section]

 \newtheorem{definition}{Definition}[section]

\begin{document}
 \pagenumbering{arabic} \thispagestyle{empty}
\setcounter{page}{1}

\begin{center}
\begin{center}{\Large{\bf Tail Asymptotics for a Retrial Queue with Bernoulli Schedule}}

\vskip 1cm

 { Bin Liu 
 \\ {\small School of Mathematics \& Physics, Anhui Jianzhu University,
Hefei 230601, P.R. China, \\
bliu@amt.ac.cn}\\

\

Yiqiang Q. Zhao

{\small School of Mathematics and Statistics,
Carleton University, Ottawa, ON, Canada K1S 5B6,\\
zhao@math.carleton.ca}}
\end{center}

\date{\today}

\end{center}

\begin{abstract}
In this paper, we study the asymptotic behaviour of the tail probability of the number of customers in the steady-state $M/G/1$ retrial queue with Bernoulli schedule, under the assumption that the service time distribution has a regularly varying tail. Detailed tail asymptotic properties are obtained for the conditional probability of the number of customers in the (priority) queue and orbit, respectively.
\medskip

\end{abstract}


\

\section{Introduction}

As one of the important types of queueing systems, retrial queues have been extensively studied for more than 40 years and have produced more than a 100 literature publications. Research on retrial queues is still very active and new challenges are continuing to emerge.
A general picture of retrial models, together with their applications in various areas, and basic results on retrial queues can be acquired from books or recent surveys, such as
Falin and Templeton~\cite{Falin-Templeton:1997},
Artajelo and G\'{o}mez-Corral~\cite{Artalejo-GomezCorral:2008},
Choi and Chang~\cite{Choi-Chang:1999},
Kim and Kim~\cite{Kim-Kim:2016}, Phung-Duc~\cite{Phung-Duc:2017}, among possible others.

Tail asymptotic analysis of retrial queueing systems, especially asymptotic properties in tail stationary probabilities for a stable retrial queue, have been a focus of the investigation in the past 10 years or so, due to two main reasons: first, for most of retrial queues, they are not expected to have explicit non-Laplace transform solutions for their stationary distributions, and second, tail asymptotic properties often lead to approximations to performance metrics and numerical algorithms. Both light-tailed and heavy-tailed properties have been obtained for a number of retrial queues, including the following incomplete list: Kim, Kim and Ko~\cite{Kim-Kim-Ko:2007},
Liu and Zhao~\cite{Liu-Zhao:2010},
Kim, Kim and Kim~\cite{Kim-Kim-Kim:2010a},
Liu, Wang and Zhao~\cite{Liu-Wang-Zhao:2012},
Kim and Kim~\cite{Kim-Kim:2012},
Artalejo and Phung-Duc~\cite{Artalejo-PhungDuc:2013},
Walraevens, Claeys and Phung-Duc~\cite{Walraevens-Claeys-Phung-Duc:2018},
Kim, Kim and Kim~\cite{Kim-Kim-Kim:2010c},
Yamamuro~\cite{Yamamuro:2012},
Liu, Wang and Zhao~\cite{Liu-Wang-Zhao:2014},
Masuyama~\cite{Masuyama:2014},
Liu, Min and Zhao~\cite{Liu-Min-Zhao:2017}, and
Liu and Zhao~\cite{Liu-Zhao:2017b}.

In this paper, we consider an $M/G/1$ retrial queue with Bernoulli schedule. This model was first proposed and studied by Choi and Park~\cite{Choi-Park:1990}.
Specifically, this $M/G/1$ retrial queueing system has a (priority) queue of infinite waiting capacity and an orbit.  External customers arrive to this system according to a Poisson process with rate $\lambda$. There is a single server in this system. If the server is idle upon the arrival of a customer, the customer receives the service immediately and leaves the system after the completion of the service. Otherwise, if the server is busy, the arriving customer would join the queue with probability $q$, becoming a priority customer waiting for the service according to the first-in-first-out discipline; or the orbit with probability $p=1-q$, becoming a repeated customer who will retry later for receiving its service. A priority customer has priority over a repeated customer for receiving the service. This implies that upon the completion of a service, if there is a customer in the queue, the server will serve the customer at the head of the queue; otherwise the server becomes idle. Each of the repeated customers in the orbit independently repeatedly tries to receive service, according to a Poisson process with the retrial rate $\mu$, until they find an idle server. Once an idle server is found, they immediately receive the service and, after the completion of the service, they leave the sytsem. All customers receive the same service time whose distribution $F_{\beta}(x)$ with $F_{\beta}(0)=0$ is assumed to have a finite mean $\beta_1$. The Laplace-Stieltjes transform (LST) of the distribution function $F_{\beta}(x)$ is denoted by $\beta(s)$.

Closely related to the model studied in \cite{Choi-Park:1990} and also in this paper, are
Falin, Artalejo and Martin~\cite{Falin-Artalejo-Martin:1993}, in which a model with two independent (primary) Poisson arrival streams was considered. The priority customers, when blocked upon arrival, are queued and waiting for service, while the non-priority customers, when blocked, join the orbit and retry for service later;
Li and Yang~\cite{Li-Yang:1998}, in which the discrete-time counterpart to the model studied in this paper, or a Geo/G/1 retrial queue with Bernoulli schedule, was considered;
and Atencia and Moreno~\cite{Atencia-Moreno:2005}, in which the $M/G/1$ retrial queueing system with Bernoulli schedule has a general retrial time, but only the customer at the head of the orbit is allowed to retry for service, or retrials with a constant rate.

Let $\lambda_1=\lambda q$, $\lambda_2=\lambda p$, $\rho_1=\lambda_1\beta_1$, $\rho_2=\lambda_2\beta_1$ and $\rho=\rho_1+\rho_2=\lambda\beta_1$. It follows from \cite{Choi-Park:1990} that the system considered in this paper is stable if and only if (iff) $\rho<1$, which is assumed to be true throughout the paper.
Our focus is on asymptotic properties for various stationary tail probabilities of the $M/G/1$ retrial queue with Bernoulli schedule, which have not been studied before.
The starting point of our study is based on two expressions for probability transformations obtained in \cite{Choi-Park:1990}. Stochastic decompositions are tools in our analysis.
By assuming a regular varying tail in the service time distribution, we obtain, in this paper, asymptotic properties for the conditional tail probabilities of customers:
(1) in the orbit, given that the server is idle (Section~\ref{R0+Mc});
(2) in the queue, given that the server is busy (Section~\ref{sec:4.2});
(3) in the orbit, given that the server is busy (Section~\ref{sec:4.3});
(4) in the orbit, given that the server is busy and the queue is empty (Section~\ref{sec:4.6}).

The rest of this paper is organized as follows:
preliminary results are provided in Section~\ref{sec:2};
stochastic decompositions are obtained in Section~\ref{sec:3};
and the main results, asymptotic properties for tail probabilities, are derived in Section~\ref{sec:4}.

\section{Preliminary} \label{sec:2}

Assume that the system is in steady state. Let $R_{que}$ be the number of priority customers in the queue, {\it excluding} the possible one in the service, let $R_{orb}$ be the number of repeated customers in the orbit, and let $I_{ser}=1$ or 0, whenever the server is busy or idle, respectively. Let $R_0$ be a random variable (r.v.) whose distribution coincides with the conditional distribution of $R_{orb}$, given that $I_{ser}=0$, and let $(R_{11},R_{12})$ be a two-dimensional r.v. whose distribution coincides with the conditional distribution of $(R_{que},R_{orb})$, given that $I_{ser}=1$. Precisely, $R_0$, $R_{11}$ and $R_{12}$ are all nonnegative integer-valued r.v.s; $R_0$ has the probability generating function (PGF): $R_0(z_2)=E(z_2^{R_0})\overset{\footnotesize\mbox{def}}{=}E(z_2^{R_{orb}}|I_{ser}=0)$ and $(R_{11},R_{12})$ has the PGF: $R_1(z_1,z_2)=E(z_1^{R_{11}}z_2^{R_{12}})\overset{\footnotesize\mbox{def}}{=}E(z_1^{R_{que}}z_2^{R_{orb}}|I_{ser}=1)$.

Our starting point for tail asymptotic analysis is based on the expressions for $R_0(z_2)$ and $R_1(z_1,z_2)$. Following the discussions in~\cite{Choi-Park:1990}, let
\begin{eqnarray}
M_{a}(z_1,z_2)&=&\frac 1 {\rho}\cdot\frac {1-\beta(\lambda-\lambda_1 z_1-\lambda_2 z_2)} {1- p z_2- q z_1},\label{FalinG1}\\
M_{b}(z_1,z_2)&=&(1-\rho_1)\cdot \frac {h(z_2)-z_1} {\beta(\lambda-\lambda_1 z_1-\lambda_2 z_2)-z_1},\label{FalinG2}\\
M_{c}(z_2)&=&\frac {1-\rho} {1-\rho_1}\cdot\frac {1- z_2} {h(z_2)-z_2},\label{FalinG3}
\end{eqnarray}
where $h(\cdot)$ is determined uniquely by the following equation:
\begin{equation}
h(z)=\beta(\lambda-\lambda_1 h(z)-\lambda_2 z). \label{Falin-phi-eqn}
\end{equation}
Since $P\{I_{ser}=0\}=1-\rho$ and $P\{I_{ser}=1\}=\rho$, are obtained in \cite{Choi-Park:1990}, we have the following expressions immediately from equations (12) and (13) in~\cite{Choi-Park:1990}:
\begin{eqnarray}
R_0(z_2)&=&\exp\left\{-\frac {\lambda} {\mu}\int_{z_2}^1\frac {1-h(u)} {h(u)-u} du\right\},\label{FalinD0}\\
R_1(z_1,z_2)&=&M_a(z_1,z_2)\cdot M_b(z_1,z_2)\cdot M_c(z_2)\cdot R_0(z_2).\label{FalinD1}
\end{eqnarray}

In the next section, we will prove that $R_0(z_2)$, $M_c(z_2)$, $M_a(z_1,z_2)$ and $M_b(z_1,z_2)$ can be viewed as the PGFs of four one or two-dimensional r.v.s.

Next, we provide a probabilistic interpretation for $h(z)$.
Let $T_{\alpha}$ be the busy period of the standard $M/G/1$ queue with arrival rate $\lambda_1$ and service time $T_{\beta}$.
By $F_{\alpha}(x)$, we denote the probability distribution function of $T_{\alpha}$, and by $\alpha (s)$, the LST of $F_{\alpha}(x)$.
The following are classic results on the busy period of the standard $M/G/1$ queue (e.g., \cite{Meyer-Teugels:1980}):
\begin{eqnarray}
\alpha(s)&=&\beta(s+\lambda_1- \lambda_1 \alpha(s)),\label{busy-eqn-alpha}\\
\alpha_1&\stackrel{\rm def}{=}&E(T_{\alpha})=\beta_1/(1-\rho_1).\label{busy-exp-T-alpha}
\end{eqnarray}

Throughout this paper, we will use the notation $N_b(t)$ to represent the number of Poisson arrivals, with rate $b$, within the time interval $(0,t]$,
and $N_{\lambda_2}(T_{\alpha})$ to represent the number of arrivals of a Poisson process, with arrival rate $\lambda_2$, within the independent random time  $T_{\alpha}$. The PGF of $N_{\lambda_2}(T_{\alpha})$ is easily obtained as follows:
\begin{equation}
E(z^{N_{\lambda_2}(T_{\alpha})})=\int_0^{\infty}\sum_{n=0}^{\infty}z^n\frac {(\lambda_2 x)^n} {n!} e^{-\lambda_2 x}dF_{\alpha}(x)=\alpha(\lambda_2-\lambda_2 z).
\end{equation}
It follows from (\ref{busy-eqn-alpha}) that
\begin{equation}
\alpha(\lambda_2-\lambda_2 z)=\beta(\lambda - \lambda_1 \alpha(\lambda_2-\lambda_2 z)-\lambda_2 z ).\label{busy-eqn-alpha-z}
\end{equation}
By comparing (\ref{Falin-phi-eqn}) and (\ref{busy-eqn-alpha-z}) and noticing the uniqueness of $h(z)$, we immediately have
\begin{equation}
h(z)=\alpha(\lambda_2-\lambda_2 z)=E(z^{N_{\lambda_2}(T_{\alpha})}).\label{alpha-phi}
\end{equation}

\begin{remark}\label{h(z)}
$h(z)$ is the PGF of the number of arrivals of a Poisson process with arrival rate $\lambda_2$ within an independent random time $T_{\alpha}$, where $T_{\alpha}$ has the same probability distribution as that for
the busy period of the standard $M/G/1$ queue with arrival rate $\lambda_1$ and service time $T_{\beta}$.
\end{remark}

We will use $L(t)$ to represent a slowly varying function at $\infty$ (see Definition~\ref{Definition 3.1}) and make the following basic assumption on the service time $T_{\beta}$:
\\
\\
\noindent{\bf A.}
{\it The service time $T_{\beta}$ has tail probability
$P\{T_{\beta}>t\} \sim t^{-a}L(t)$ as $t\to\infty$, where $a>1$.
}

Since $T_{\alpha}$ is the busy period of the ordinary $M/G/1$ queue with arrival rate $\lambda_1$ and the service time $T_{\beta}$, its asymptotic tail probability is regularly varying according to de Meyer and Teugels~\cite{Meyer-Teugels:1980} (see Lemma~\ref{Lemma 2.1} in Appendix).

\section{Stochastic decompositions} \label{sec:3}

In this section, we will verify that $R_0(z_2)$, $M_a(z_1,z_2)$, $M_b(z_1,z_2)$ and $M_c(z_2)$ can be viewed as the PGFs of four r.v.s, and at the same time, we will find the stochastic decompositions of these r.v.s, which will be used for the asymptotic analysis in later sections.

\subsection{Probabilistic interpretation for PGF $R_0(z_2)$}

Based on the definition of $R_0$ given earlier, this is a r.v. whose distribution coincides with the conditional distribution of $R_{orb}$, given that $I_{ser}=0$. In this section, we provide a new probabilistic interpretation for $R_0$, which is useful for our tail asymptotic analysis.

Substituting (\ref{alpha-phi}) into (\ref{FalinD0}), we obtain
\begin{equation}
R_0(z_2)=\exp\left\{-\frac {\lambda} {\mu}\int_{z_2}^1\frac {1-\alpha(\lambda_2-\lambda_2 u)} {\alpha(\lambda_2-\lambda_2 u)-u} du\right\}.\label{D0-2}
\end{equation}
We now rewrite (\ref{D0-2}). Let
\begin{eqnarray}
\psi &=&\frac{\rho}{\mu(1-\rho)}, \label{psi} \\
\kappa(s)&=&\frac{1-\rho} {\beta_1}\cdot\frac{1-\alpha(s)} {s-\lambda_2+\lambda_2\alpha(s)},\label{K(u)}\\
\omega(s)&=&1-\int_0^s\kappa(u)du,\label{omega(s)}\\
\tau(s)&=&\exp\left\{\psi\omega(s)-\psi\right\}.\label{tau00}
\end{eqnarray}
It is easy to see, from (\ref{D0-2})--(\ref{tau00}), that
\begin{eqnarray}
    R_0(z_2)&=& \exp\left\{-\frac {\lambda} {\mu}\int_0^{\lambda_2-\lambda_2 z_2}\frac{1-\alpha(s)} {s-\lambda_2+\lambda_2\alpha(s)}ds\right\}\nonumber\\
    &=&\exp\left\{-\psi\int_0^{\lambda_2-\lambda_2 z_2}\kappa(u)du\right\}\nonumber\\
    &=&\tau(\lambda_2-\lambda_2 z_2).\label{D^{(0)}}
\end{eqnarray}

In the following Remark~\ref{kappa-geo-sum}, Remark~\ref{omega-LST} and Remark~\ref{tau-poisson-sum}, we assert that $\kappa(s)$, $\omega(s)$ and $\tau(s)$ are the LSTs of three probability distribution functions on $[0,\infty)$, respectively. For the first assertion, let $F_{\alpha}^{(e)}(x)$ be the equilibrium distribution of $F_{\alpha}(x)$,
which is defined as $F_{\alpha}^{(e)}(x)= \alpha_1^{-1}\int_0^{x}(1-F_{\alpha}(t))dt$ where $\alpha_1=E(T_{\alpha})$ is given in (\ref{busy-exp-T-alpha}).
The LST of $F_{\alpha}^{(e)}(x)$ can be written as $\alpha^{(e)}(s)=(1-\alpha(s))/(\alpha_1 s)$.
From (\ref{K(u)}), we have
\begin{equation}\label{kappa(s)-2}
\kappa(s) = \frac{(1-\vartheta)\alpha^{(e)}(s)}{1-\vartheta\alpha^{(e)}(s)}=\sum_{k=1}^{\infty}(1-\vartheta)\vartheta^{k-1}(\alpha^{(e)}(s))^k,
\end{equation}
where
\begin{equation}\label{vartheta}
\vartheta=\lambda_2\alpha_1=\rho_2/(1-\rho_1)<1.
\end{equation}

\begin{remark}\label{kappa-geo-sum}
Immediately from (\ref{kappa(s)-2}), $\kappa(s)$ can be viewed as the LST of the distribution function of a r.v. $T_{\kappa}$, whose distribution function is denoted by $F_{\kappa}(x)$. More precisely, $T_{\kappa}\stackrel{\rm d}{=}T_{\alpha,1}^{(e)}+T_{\alpha,2}^{(e)}+\cdots+T_{\alpha,J}^{(e)}$,
where $T_{\alpha,j}^{(e)}$, $j\ge 1$ are i.i.d. r.v.s, each with the distribution $F_{\alpha}^{(e)}(x)$, $P(J=j)=(1-\vartheta)\vartheta^{j-1}$, $j\ge 1$ and $J$ is independent of $T_{\alpha,j}^{(e)}$ for $j\ge 1$.
In the above, the notation ``$\stackrel{\rm d}{=}$'' was used (which will be used later again) to  mean the equality in probability distribution.
\end{remark}

For the second assertion,  we will use Theorem~1 in Feller~\cite{Feller1971} (see p.439): A function $\varphi(s)$ is the LST of a probability distribution function iff $\varphi(0)=1$ and $\varphi(s)$ is completely monotone, i.e., $\varphi(s)$ possesses derivatives $\varphi^{(n)}(s)$ of all orders such that $(-1)^{n}\varphi^{(n)}(s)\ge 0$ for $s> 0$.
\begin{remark}\label{omega-LST}
Immediately from (\ref{omega(s)}), $\omega(0)=1$ and $(-1)^{n}\omega^{(n)}(s)=(-1)^{n-1}\kappa^{(n-1)}(s)\ge 0$ for $s> 0$, and we assert that
$\omega(s)$ is the LST of a probability distribution function of a r.v. $T_{\omega}$, whose distribution is denoted by $F_{\omega}(x)$.
\end{remark}

For the third assertion, we write (\ref{tau00}) as
\begin{equation}\label{tau(s)-2}
\tau(s) = \sum_{k=0}^{\infty}\frac {\psi^k} {k!} e^{-\psi} (\omega(s))^k.
\end{equation}

\begin{remark}\label{tau-poisson-sum}
Immediately from (\ref{tau(s)-2}), $\tau(s)$ can be viewed as the LST of the distribution function of a r.v. $T_{\tau}$, whose distribution function is denoted by $F_{\tau}(x)$. More precisely, $T_{\tau}\stackrel{\rm d}{=}\sum_{j=0}^J T_{\omega,j}$, where $T_{\omega,j}$, $j\ge 1$, are i.i.d. r.v.s  each with the distribution $F_{\omega}(x)$, and $P(J=j)=\frac {\psi^j} {j!} e^{-\psi}$, $j\ge 0$, where $J$ is independent of $T_{\omega,j}$ for $j\ge 0$.
\end{remark}

The following remark provides the detailed interpretation on (\ref{D^{(0)}}).

\begin{remark}\label{tau-compound-possion}
$R_0$ can be regarded as the number of Poisson arrivals with rate $\lambda_2$ within an independent random time $T_{\tau}$, i.e.,
$R_0\stackrel{\rm d}{=}N_{\lambda_2}(T_{\tau})$.
\end{remark}

\subsection{Probabilistic interpretation for PGF $M_c(z_2)$}

Recall that $M_c(z_2)$ is defined in (\ref{FalinG3}). In this section, we show that $M_c(z_2)$ is
 the PGF of a r.v., denoted by $M_c$, or $M_c(z_2)=E(z_2^{M_c})$, and provide a probabilistic interpretation for $M_c$.

It follows from (\ref{FalinG3}) and (\ref{alpha-phi}) that
\begin{eqnarray}\label{M_c(z_2)}
M_c(z_2)&=&\frac {(1-\vartheta)\cdot(1- z_2)} {\alpha(\lambda_2-\lambda_2 z_2)-z_2}
=\frac {1-\vartheta} {1-\vartheta\alpha^{(e)}(\lambda_2-\lambda_2 z_2)}\nonumber\\
&=&\sum_{k=0}^{\infty} (1-\vartheta) \vartheta^{k} (\alpha^{(e)}(\lambda_2-\lambda_2 z_2) )^{k}\nonumber\\
&=&1-\vartheta+\vartheta\cdot\kappa(\lambda_2-\lambda_2 z_2)\quad \mbox{(by (\ref{kappa(s)-2}))}.
\end{eqnarray}
Define
\begin{equation}\label{def-T-theta}
T_{\eta}=\left\{\begin{array}{ll}
0, &\mbox{ with probability }1-\vartheta,\\
T_{\kappa}, &\mbox{ with probability }\vartheta.
\end{array}
\right.
\end{equation}
Denote by $\eta(s)$ the LST of the probability distribution function of $T_{\eta}$. We immediately see that $\eta(s)=1-\vartheta+\vartheta\cdot\kappa(s)$, which, together with (\ref{M_c(z_2)}), yields
\begin{equation}\label{M_c(z_2)-2}
M_c(z_2)=\eta(\lambda_2-\lambda_2 z_2).
\end{equation}

The following remark provides a probabilistic interpretation for $M_c$.
\begin{remark}\label{remark3-3}
With the same argument as that in Remark \ref{tau-compound-possion}, (\ref{M_c(z_2)-2}) allows us to interpret the r.v. $M_c$ as the number of Poisson arrivals with rate $\lambda_2$ within an independent random time $T_{\eta}$, i.e., $M_c\stackrel{\rm d}{=}N_{\lambda_2}(T_{\eta})$.
\end{remark}

\subsection{Probabilistic interpretation for PGF $M_a(z_1,z_2)$}

In this section, we prove that $M_a(z_1,z_2)$, defined in (\ref{FalinG1}), is the PGF of a two-dimensional r.v., denoted by $(M_{a1},M_{a2})$, or
$M_a(z_1,z_2)=E (z_1^{M_{a1}}z_2^{M_{a2}} )$. Denote $F_{\beta}^{(e)}(x)$ to be the equilibrium distribution of $F_{\beta}(x)$, that is,
$F_{\beta}^{(e)}(x)=\beta_1^{-1} \int_0^{x} (1-F_{\beta}(t))dt.$
The LST of $F_{\beta}^{(e)}(x)$ can be written as
$\beta^{(e)}(s)= (1-\beta(s))/(\beta_1 s).$
It follows from (\ref{FalinG1}) that
\begin{equation}
M_a(z_1,z_2)=\beta^{(e)}(\lambda-\lambda_1 z_1 -\lambda_2 z_2).\label{M1(z1,z2)}
\end{equation}

For simplicity, we introduce the following concept:
\begin{definition}
Let $N$ be a non-negative integer valued r.v. and let $\{X_k\}_{k=1}^{\infty}$ be a sequence of i.i.d. Bernoulli r.v.s having a common distribution $P\{X_k=1\}=c$ and $P\{X_k=0\}=1-c$, where  $0<c<1$, and are assumed to be independent of $N$. The two-dimensional r.v.
$(\sum_{k=1}^N X_k,N-\sum_{k=1}^N X_k)$, where $\sum_{1}^0\equiv 0$,
is called an independent $(c,1-c)$-splitting of $N$, denoted by $split(N;c,1-c)$.
\end{definition}

The following remark provides a definition for $(M_{a1},M_{a2})$, together with its probabilistic interpretation.
\begin{remark}\label{M{a1},M{a2}}
$(M_{a1},M_{a2})\stackrel{\rm d}{=}split(N_{\lambda}(T_{\beta}^{(e)});q,p)$, which can be easily checked because the right hand side of (\ref{M1(z1,z2)}) can be written as
$\int_0^{\infty}\sum_{n=0}^{\infty} [\sum_{k=0}^{n}z_1^{k}z_2^{n-k}\binom{n}{k}q^{k}p^{n-k} ]((\lambda x)^n/n!) e^{-\lambda x}dF_{\beta}^{(e)}(x)$. Therefore, $M_a(z_1,z_2)$ is the PGF of Poisson arrivals with rate $\lambda$, split into two components according to the independent $(c,1-c)$-splitting defined above, within an independent random time $T^{(e)}_{\beta}$ whose distribution coincides with the equilibrium distribution $F_{\beta}^{(e)}(x)$.
\end{remark}

\subsection{Probabilistic interpretation for PGF $M_b(z_1,z_2)$}

In this section, we prove that $M_b(z_1,z_2)$, defined in (\ref{FalinG2}) is the PGF of a two-dimensional r.v., denoted by $(M_{b1},M_{b2})$, or $M_b(z_1,z_2)= E(z_1^{M_{b1}}z_2^{M_{b2}} )$. Let
\begin{eqnarray}
H(z_1,z_2)&=&\frac 1 {\rho_1}\cdot\frac {\beta(\lambda-\lambda_1 z_1-\lambda_2 z_2)-h(z_2)} {z_1-h(z_2)} \label{G2-decomp-H}\\
&=&\frac 1 {\rho_1}\cdot\frac {\beta(\lambda-\lambda_1 z_1-\lambda_2 z_2)-\beta(\lambda-\lambda_1 h(z_2)-\lambda_2 z_2)} {z_1-h(z_2)}\quad \mbox{(by (\ref{Falin-phi-eqn}))}.\label{G2-decomp-H-x}
\end{eqnarray}
It follows from (\ref{FalinG2}) and (\ref{G2-decomp-H}) that
\begin{equation}
M_b(z_1,z_2)=\frac {1-\rho_1} {1-\rho_1 H(z_1,z_2)}=\sum_{n=0}^{\infty} (1-\rho_1) \rho_1^{n} \left(H(z_1,z_2)\right)^n, \label{G2-decomp-eqn1}
\end{equation}
Clearly,
$(M_{b1},M_{b2})$ can be regarded as a geometric sum of two-dimensional r.v.s., provided that $H(z_1,z_2)$ is the PGF of a two-dimensional r.v.
To verify this, we write (\ref{G2-decomp-H-x}) as a power series.
Let $b_k=\int_0^{\infty}\frac{(\lambda t)^k} {k!}e^{-\lambda t} dF_{\beta}(t)$, $k\ge 1$.
Hence,
\begin{equation}
\beta(\lambda-\lambda_1 z_1-\lambda_2 z_2)=\int_0^{\infty}\sum_{k=0}^{\infty}\frac{(\lambda(q z_1+p z_2)t)^k} {k!}\cdot e^{-\lambda t} dF_{\beta}(t)=\beta(\lambda)+\sum_{k=1}^{\infty} b_k (q z_1+p z_2)^k, \label{H-lemma-1-proof-1}
\end{equation}
and
\begin{equation}
\beta(\lambda-\lambda_1 h(z_2)-\lambda_2 z_2)=\beta(\lambda)+\sum_{k=1}^{\infty} b_k (q h(z_2)+p z_2)^k. \label{H-lemma-1-proof-2}
\end{equation}
Substituting (\ref{H-lemma-1-proof-1}) and (\ref{H-lemma-1-proof-2}) into the numerator of the right-hand side of (\ref{G2-decomp-H-x}), we obtain
\begin{eqnarray}
H(z_1,z_2)
&=&\frac 1 {\rho_1}\cdot\sum_{k=1}^{\infty} b_k \left(\frac{(q z_1+p z_2)^k-(q h(z_2)+p z_2)^k} {z_1-h(z_2)}\right)\nonumber\\
&=&\frac 1 {\rho}\cdot\sum_{k=1}^{\infty} b_k \sum_{i=1}^{k}(q z_1+p z_2)^{i-1}(q h(z_2)+p z_2)^{k-i}\nonumber\\
&=&\frac 1 {\rho}\cdot\sum_{k=1}^{\infty} k b_k\cdot D_{k}(z_1,z_2),\label{H-lemma-1-proof-3}
\end{eqnarray}
where
\begin{eqnarray}\label{Hj-lemma-1-eqn-1}
D_{k}(z_1,z_2)&=&\frac 1 k \sum_{i=1}^{k}(q z_1+p z_2)^{i-1}(q h(z_2)+p z_2)^{k-i}.
\end{eqnarray}
Note that $q h(z_2)+p z_2$ and $q z_1+p z_2$ are the PGFs of r.v.s (one or two-dimensional). Hence, for $k\ge 1$, $D_{k}(z_1,z_2)$ is the PGF of a two-dimensional r.v., denoted by $(D_{k,1},D_{k,2})$. In addition,
\begin{equation}
\sum_{k=1}^{\infty} k b_k=\sum_{k=1}^{\infty} \int_0^{\infty}\frac{(\lambda t)^k} {(k-1)!}e^{-\lambda t}dF_{\beta}(t)=\lambda\int_0^{\infty}tdF_{\beta}(t)=\rho.\label{sum-bnbar}
\end{equation}
Namely, $(1/\rho)(\sum_{k=1}^{\infty} k b_k)=1$, which together with (\ref{H-lemma-1-proof-3}) imply that $H(z_1,z_2)$ is the PGF of a two-dimensional r.v., denoted by $(H_1,H_2)$. Namely, $H(z_1,z_2)=E (z_1^{H_1}z_2^{H_2})$.
The above argument is summarized in the following remarks.
\begin{remark}
Suppose that $\{(Y_{n,1},Y_{n,2})\}_{n=1}^{\infty}$ is a sequence of independent two-dimensional r.v.s, each with a common PGF $q z_1+p z_2$,  $\{Z_m\}_{m=1}^{\infty}$ is a sequence of independent r.v.s, each with a common PGF $q h(z_2)+p z_2$, and the two sequences are independent.
It follows from (\ref{Hj-lemma-1-eqn-1}) that for $k\ge 1$,
\begin{equation}
(D_{k,1},D_{k,2})\overset{\footnotesize\mbox{d}}{=}
\sum_{n=1}^{i-1}(Y_{n,1},Y_{n,2})+\sum_{m=1}^{k-i}(0,Z_{m}) \mbox{ with probability }1/k,\quad i=1,2,\cdots,k.
\end{equation}
\end{remark}
\begin{remark}\label{remark-H{1}H{2}}
It follows from (\ref{H-lemma-1-proof-3}) that
\begin{equation}
(H_1,H_2)\overset{\footnotesize\mbox{d}}{=}
(D_{k,1},D_{k,2}) \mbox{ with probability }(1/\rho)kb_{k},\quad k\ge 1.
\end{equation}
\end{remark}

\begin{remark}\label{remark-M{b1}M{b2}}
It follows from (\ref{G2-decomp-eqn1}) that $(M_{b1},M_{b2})$ is a geometric sum of i.i.d. two-dimensional r.v.s $(H_1^{(i)},H_2^{(i)})$, $i\ge 1$, each with the same PGF $H(z_1,z_2)$, and precisely,
\begin{equation}\label{remark-M{b1}M{b2}-eqn1}
(M_{b1},M_{b2})\stackrel{\rm d}{=}\left\{\begin{array}{ll}
(0,0), &\mbox{ with probability }1-\rho_1,\\
\sum_{i=1}^{J}(H_1^{(i)},H_2^{(i)}) &\mbox{ with probability }\rho_1,
\end{array}
\right.
\end{equation}
where $P(J=n)=(1-\rho_1)\rho_1^{n-1}$, $n\ge 1$ and $J$ is independent of $(H_1^{(i)},H_2^{(i)})$ for $i\ge 1$.
\end{remark}

\section{Tail Asymptotics} \label{sec:4}

In this section, we will study the asymptotic behavior for the tail probabilities
$P\{R_{orb}>j|I_{ser}=0\}$, $P\{R_{que}>j|I_{ser}=1\}$, $P\{R_{orb}>j|I_{ser}=1\}$
and $P\{R_{orb}>j|R_{que}=0,I_{ser}=1\}$, as $j\to\infty$, respectively.

\subsection{Asymptotic tail probability for $P\{R_{orb}>j|I_{ser}=0\}$} \label{R0+Mc}

To study the asymptotic behavior for the tail probability $P\{R_{orb}>j|I_{ser}=0\}\equiv P\{R_0>j\}$,
let us first study the asymptotic properties of the tail probabilities for $T_{\kappa}$, $T_{\omega}$ and $T_{\tau}$, respectively.
By Lemma \ref{Lemma 2.1}, we know that $P\{T_{\alpha}>t\}\sim (1-\rho_1)^{-a-1} t^{-a} L(t)$ as $t\to\infty$, where the r.v. $T_\alpha$ is the busy period defined in Section~\ref{sec:2}.
Applying Karamata's theorem (e.g., p.28 in \cite{Bingham:1989}), we have $\int_t^{\infty}(1-F_{\alpha}(x))dx\sim (a-1)^{-1} (1-\rho_1)^{-a-1} t^{-a+1} L(t)$,
which implies that $1-F_{\alpha}^{(e)}(t) \sim ((a-1)\alpha_1)^{-1} (1-\rho_1)^{-a-1} t^{-a+1} L(t)$, $t\to\infty$.
By Remark~\ref{kappa-geo-sum} and applying Lemma~\ref{Embrechts-compound-geo}, we have
\begin{eqnarray}\label{P{T-kappa>t}}
P\{T_{\kappa}>t\}&\sim& c_{\kappa} \cdot t^{-a+1} L(t),\quad t\to\infty,
\end{eqnarray}
where
\begin{eqnarray}\label{c-kappa}
c_{\kappa}=\frac 1 {\alpha_1 (1-\vartheta)(a-1)(1-\rho_1)^{a+1}}=\frac 1 {\beta_1 (1-\rho)(a-1)(1-\rho_1)^{a-1}}.
\end{eqnarray}

In the following lemma, we present the asymptotic tail probability of $T_{\omega}$.

\begin{lemma}\label{theorem-T-tail}
\begin{equation}
P\{T_{\omega}>t\} \sim (1-1/a)c_{\kappa}\cdot t^{-a}L(t),\quad t\to\infty.\label{main}
\end{equation}
\end{lemma}

\proof
Recall that $T_{\omega}$ has the distribution function $F_{\omega}(x)$, defined in terms of its LST $\omega(s)$ in (\ref{omega(s)}),
which is determined by the LST $\kappa(s)$ of the distribution function of $T_{\kappa}$. We divide the proof into two parts, depending on whether $a$ is an integer or not.

\noindent{\it Case 1: Non-integer $a >1$.}
Suppose that $m<a<m+1$, $m\in\{1,2,\ldots\}$. Since $P\{T_{\kappa}>t\}\sim c_{\kappa}\cdot t^{-a+1}L(t)$, we have its moments $\kappa_{m-1}<\infty$ and $\kappa_{m}=\infty$.
Define $\kappa_{m-1}(s)$ in a manner similar to that in (\ref{phi1}). Then $\kappa(s) = \sum_{k=0}^{m-1}\frac{\kappa_k}{k!}(-s)^k+(-1)^{m}\kappa_{m-1}(s)$.
By Lemma~\ref{Cohen}, we know that
\begin{equation}\label{kappa{m-1}(s)asym}
\kappa_{m-1}(s) \sim \big[\Gamma(a-m)\Gamma(m+1-a)/\Gamma(a-1)\big]\cdot c_{\kappa} s^{a-1}L(1/s),\quad s\downarrow 0.
\end{equation}
From (\ref{omega(s)}), there are constants $\{v_k;\ k=0,1,2,\ldots,m\}$ satisfying $\omega(s) = \sum_{k=0}^{m}v_k (-s)^k+(-1)^{m+1} \int_0^s\kappa_{m-1}(u)du$.
Define $\omega_m(s)$ in a manner similar to that in (\ref{phi1}). Then
\begin{equation}
\omega_m(s) = \int_0^s\kappa_{m-1}(u)du \sim \big[\Gamma(a-m)\Gamma(m+1-a)/\Gamma(a)\big] \cdot\frac {a-1} {a} c_{\kappa} s^{a}L(1/s),\quad s\downarrow 0,\label{case1result}
\end{equation}
where we have used (\ref{kappa{m-1}(s)asym}) and Karamata's theorem (p.28 in \cite{Bingham:1989}). Applying Lemma \ref{Cohen}, we
complete the proof of Theorem \ref{theorem-T-tail} for non-integer $a>1$.
\\
\\
{\it Case 2: Integer $a >1$.}
Suppose that $a =m\in\{2,3,\ldots\}$. Since $P\{T_{\kappa}>t\}\sim c_{\kappa}\cdot t^{-m+1}L(t)$ implies that $T_{\kappa}$ has its moment $\kappa_{m-2}<\infty$, we can define $\kappa_{m-2}(s)$ and $\widehat{\kappa}_{m-2}(s)$ in a manner similar to that in (\ref{phi1}) and (\ref{phi2}), respectively.
Then $\kappa(s) = \sum_{k=0}^{m-2}\frac{\kappa_k}{k!}(-s)^k+(-1)^{m-1}\kappa_{m-2}(s)$.
By Lemma~\ref{Lemma 4.5new}, we obtain
\begin{equation}
\widehat{\kappa}_{m-2}(xu)-\widehat{\kappa}_{m-2}(u)\sim -(\log x) c_{\kappa} L(1/u)/(m-2)!\quad \mbox{ as}\ u\downarrow 0.\label{case2kappa-5}
\end{equation}
From (\ref{omega(s)}), there exist constants $\{v_k;\ k=0,1,2,\ldots,m\}$ satisfying $\omega(s) = \sum_{k=0}^{m}v_k (-s)^k+(-1)^{m} \int_0^s\kappa_{m-2}(u)du.$
Define $\widehat{\omega}_{m-1}(s)$ in a manner similar to that in (\ref{phi2}). Then, we have
\begin{equation}
\widehat{\omega}_{m-1}(s) = v_m+ \frac 1 {s^m}\int_0^s u^{m-1}\widehat{\kappa}_{m-2}(u)du,\label{case2tau-2}
\end{equation}
which immediately gives us $\widehat{\omega}_{m-1}(xs)=v_m+ \frac 1 {s^m}\int_0^{s} u^{m-1}\widehat{\kappa}_{m-2}(xu)du$.
It follows that
\begin{eqnarray}
\widehat{\omega}_{m-1}(xs)-\widehat{\omega}_{m-1}(s)&=&\frac 1 {s^m}\int_0^{s} u^{m-1}\big[\widehat{\kappa}_{m-2}(xu)-\widehat{\kappa}_{m-2}(u)\big]du\nonumber\\
&\sim& - \frac {m-1} {m} c_{\kappa} \log x \frac {L(1/s)} {(m-1)!}\quad \mbox{ as}\ s\downarrow 0,\label{case2kappa-5add}
\end{eqnarray}
where we have used (\ref{case2kappa-5}) and Karamata's theorem (p.28 in \cite{Bingham:1989}).
By applying Lemma~\ref{Lemma 4.5new}, we
complete the proof of Theorem \ref{theorem-T-tail} for integer $a=m\in \{2,3,\ldots\}$.
\pend

Recall Remark \ref{tau-poisson-sum}, with $T_{\tau}\stackrel{\rm d}{=}\sum_{j=0}^J T_{\omega,j}$, where $J$ has a Poisson distribution with parameter $\psi$,
$T_{\omega,j}$, $j\ge 1$, are i.i.d. r.v.s with the distribution function $F_{\omega}(x)$, and $J$ is independent of $T_{\omega,j}$, $j\ge 1$.
Therefore, by Lemma~\ref{Embrechts-compound-geo},
\begin{equation}\label{P{T-tau>t}}
P\{T_{\tau}>t\} \sim \psi\cdot P\{T_{\omega}>t\} \sim (1-1/a)c_{\kappa}\psi\cdot t^{-a}L(t),\quad t\to\infty.
\end{equation}

From Remark \ref{tau-compound-possion}, we know that $R_0=N_{\lambda_2}(T_{\tau})$. Lemma \ref{lemma-Asmussen-poisson} leads to
\begin{equation}\label{P{R_0>j}sim}
P\{R_0>j\}\sim P\{T_{\tau}>j/\lambda_2\} \sim
\frac {\lambda \lambda_2^a} {a\mu(1-\rho)^2 (1-\rho_1)^{a-1}}\cdot j^{-a}L(j),\quad j\to\infty.
\end{equation}

\subsection{Asymptotic tail probability for $P\{R_{que}>j|I_{ser}=1\}$} \label{sec:4.2}

In this subsection, we will study the asymptotic behaviour of the tail probability $P\{R_{que}>j|I_{ser}=1\}\equiv P\{R_{11}>j\}$ as $j\to\infty$.
Note that $E(z_1^{R_{11}})=R_1(z_1,1)$. Taking $z_2\to 1$ in (\ref{FalinG3}), (\ref{FalinD0}) and (\ref{FalinD1}) and using the fact that $\frac d {dz} h(z)\big |_{z=1}=\lambda_2 \alpha_1=\rho_2 /(1-\rho_1)$, we immediately have $M_c(1)=1$, $R_0(1)= 1$ and
\begin{equation}
E(z_1^{R_{11}})= M_a(z_1,1) M_b(z_1,1).\label{N_1(z_1,1)}
\end{equation}
It follows from (\ref{FalinG1}), (\ref{G2-decomp-eqn1}) and (\ref{G2-decomp-H} ) that
\begin{eqnarray}
M_a(z_1,1)&=&\frac 1 {\rho}\cdot\frac {1-\beta(\lambda_1-\lambda_1 z_1)} {q- q z_1}=\beta^{(e)}(\lambda_1-\lambda_1 z_1)\label{M_a(z_1,1)},\\
M_b(z_1,1)&=&\sum_{n=0}^{\infty} (1-\rho_1) \rho_1^n \left(H(z_1,1)\right)^n \label{M_b(z_1,1)},\\
H(z_1,1)&=&\frac 1 {\rho_1}\cdot\frac {\beta(\lambda_1-\lambda_1 z_1)-1} {z_1-1}=\beta^{(e)}(\lambda_1-\lambda_1 z_1) \label{H(z_1,1)}.
\end{eqnarray}
Substituting (\ref{H(z_1,1)}) into (\ref{M_b(z_1,1)}) and applying (\ref{M_a(z_1,1)}) and (\ref{N_1(z_1,1)}), we have
\begin{equation}\label{R11=possion-arrival}
E(z_1^{R_{11}})=\xi(\lambda_1-\lambda_1 z_1),
\end{equation}
where
\begin{equation}\label{remark-xi(s)}
\xi(s)=\sum_{n=1}^{\infty} (1-\rho_1) \rho_1^{n-1} (\beta^{(e)}(s) )^{n}.
\end{equation}
\begin{remark}\label{remark4-1}
Immediately from (\ref{remark-xi(s)}), $\xi(s)$ can be viewed as the LST of the distribution function of the r.v. $T_{\xi}\stackrel{\rm d}{=}T_{\beta,1}^{(e)}+T_{\beta,2}^{(e)}+\cdots+T_{\beta,J}^{(e)}$, where $T_{\beta,j}^{(e)}$, $j\ge 1$, are
i.i.d. r.v.s. with a common distribution $F_{\beta}^{(e)}(x)$, $P(J=j)=(1-\rho_1) \rho_1^{j-1}$, $j\ge 1$, and $J$ is independent of $T_{\beta,j}^{(e)}$, $j\ge 1$.
\end{remark}

Under Assumption A, by Karamata's theorem (e.g., p.28 in \cite{Bingham:1989}), we have $\int_t^{\infty}(1-F_{\beta}(x))dx\sim (a-1)^{-1} t^{-a+1} L(t)$,
which implies that $1-F_{\beta}^{(e)}(t) \sim ((a-1)\beta_1)^{-1}  t^{-a+1} L(t)$, $t\to\infty$.
By Remark~\ref{remark4-1} and applying Lemma~\ref{Embrechts-compound-geo}, we have
\begin{equation}\label{P{K>x}}
P\{T_{\xi}>t\} \sim \frac 1 {(1-\rho_1)(a-1)\beta_1 }\cdot t^{-a+1} L(t),\quad t\to\infty.
\end{equation}
\begin{remark}\label{R_{11}=N}
With (\ref{R11=possion-arrival}), one can interpret $R_{11}$ as the number of Poisson arrivals with rate $\lambda_1$ within the independent random time $T_{\xi}$, i.e., $R_{11}\stackrel{\rm d}{=}N_{\lambda_1}(T_{\xi})$.
\end{remark}

By Remark \ref{R_{11}=N} and applying Lemma \ref{lemma-Asmussen-poisson}, we have
\begin{equation}\label{P{R{11}>j}sim}
P\{R_{11}>j\}\sim P\{T_{\xi}>j/\lambda_1\} \sim\frac {\lambda_1^{a-1}} {(1-\rho_1)(a-1)\beta_1}\cdot j^{-a+1} L(j),\quad j\to\infty.
\end{equation}

\subsection{Asymptotic tail probability for $P\{R_{orb}>j|I_{ser}=1\}$}\label{P{R_{orb}>j|I_{ser}=1}} \label{sec:4.3}

In this subsection, we will study the asymptotic behaviour of the tail probability $P\{R_{orb}>j|I_{ser}=1\}\equiv P\{R_{12}>j\}$ as $j\to\infty$.
Note that $E(z_2^{R_{12}})=R_1(1,z_2)$. Taking $z_1\to 1$ in (\ref{FalinG1})--(\ref{FalinG3}), we have
\begin{eqnarray}
M_a(1,z_2)\cdot M_b(1,z_2)\cdot M_c(z_2)
&=&\frac {1-\rho} {\rho}\cdot\frac {1-\beta(\lambda_2-\lambda_2 z_2)} {p- p z_2}
\cdot \frac {h(z_2)-1} {\beta(\lambda_2-\lambda_2 z_2)-1}
\cdot\frac {1- z_2} {h(z_2)-z_2}\nonumber\\
&=&\frac {1-\rho} {\rho_2}\cdot \frac {1-h(z_2)} {h(z_2)-z_2}.\label{M_a(1,z_2)M_b(1,z_2)M_c(z_2)-a}
\end{eqnarray}
Substituting  (\ref{alpha-phi}) into (\ref{M_a(1,z_2)M_b(1,z_2)M_c(z_2)-a}) gives
\begin{eqnarray}
M_a(1,z_2)\cdot M_b(1,z_2)\cdot M_c(z_2)
&=&\frac {1-\rho} {\rho_2}\cdot \frac {1-\alpha(\lambda_2-\lambda_2 z_2)} {\alpha(\lambda_2-\lambda_2 z_2)-z_2}\nonumber\\
&=&\kappa(\lambda_2-\lambda_2 z_2),\label{M_a(1,z_2)M_b(1,z_2)M_c(z_2)}
\end{eqnarray}
where the last equality follows from (\ref{K(u)}).

By (\ref{FalinD1}), (\ref{M_a(1,z_2)M_b(1,z_2)M_c(z_2)}) and (\ref{D^{(0)}}),
\begin{eqnarray}\label{R12-poisson-arrival}
E(z_2^{R_{12}})&=&M_a(1,z_2)\cdot M_b(1,z_2)\cdot M_c(z_2)\cdot R_0(z_2)\nonumber\\
&=&\kappa(\lambda_2-\lambda_2 z_2)\cdot\tau(\lambda_2-\lambda_2 z_2).
\end{eqnarray}

\begin{remark}\label{K-12}
With (\ref{R12-poisson-arrival}), one can interpret $R_{12}$ as the number of Poisson arrivals with rate $\lambda_2$ within an independent random time $T_{\kappa}+T_{\tau}$, i.e., $R_{12}\stackrel{\rm d}{=}N_{\lambda_2}(T_{\kappa}+T_{\tau})$, where $T_{\kappa}$ and $T_{\tau}$ are assumed to be independent.
\end{remark}

By (\ref{P{T-kappa>t}}) and (\ref{P{T-tau>t}}), and applying Lemma \ref{Lemma 3.4new}, we have
\begin{equation}
P\{T_{\kappa}+T_{\tau}>t\}\sim P\{T_{\kappa}>t\}\sim c_{\kappa}\cdot t^{-a+1} L(t),\quad t\to\infty.
\end{equation}
Applying Remark \ref{K-12} and Lemma \ref{lemma-Asmussen-poisson}, we have
\begin{equation}\label{P{R{12}>j}sim}
P\{R_{12}>j\}\sim P\{T_{\kappa}+T_{\tau}>j/\lambda_2\}\sim \frac {\lambda_2^{a-1}} {\beta_1 (1-\rho)(a-1)(1-\rho_1)^{a-1}}\cdot j^{-a+1} L(j),\quad j\to\infty.
\end{equation}

\subsection{Asymptotic tail probability for $P\{R_{orb}>j|I_{ser}=1,R_{que}=0\}$} \label{sec:4.6}

In previous subsections \ref{R0+Mc} and \ref{P{R_{orb}>j|I_{ser}=1}}, it has been proved that $P\{R_{orb}>j|I_{ser}=0\}$ and $P\{R_{orb}>j|I_{ser}=1\}$ are regularly varying with index $-a$ and $-a+1$, respectively (the latter is one degree heavier than the former).
In this subsection, we are interested in the case in which the server is busy but the queue is empty and study the asymptotic behavior of $P\{R_{orb}>j|I_{ser}=1,R_{que}=0\}\equiv P\{R_{12}>j|R_{11}=0\}$, as $j\to\infty$.

Immediately from (\ref{FalinD1}), we are able to decomposes $(R_{11},R_{12})$ into the four independent components as follows:
\begin{equation}\label{(R{11},R{12})decomp}
(R_{11},R_{12})\stackrel{\rm d}{=}(M_{a1},M_{a2})+(M_{b1},M_{b2})+(0,M_{c})+(0,R_{0}).
\end{equation}
The tail asymptotic behaviour of $R_{0}$ has already been obtained in Subsection \ref{R0+Mc}. Below, we will focus on r.v.s. $(M_{a1},M_{a2})$, $(M_{b1},M_{b2})$ and $M_{c}$.

\subsubsection{Tail probability $P\{M_{a2}>j|M_{a1}=0\}$}

We refer to any distribution function $F(x)$ (or r.v. $X$) as light-tailed if $E(e^{\varepsilon X}) < \infty$
for some $\varepsilon > 0$, which includes, as the classic example, the exponential distribution function, as well as
all bounded r.v.s.
\begin{lemma}
Let $N$ be a discrete r.v. with an arbitrary distribution $p_{n}=P\{N=n\}$, $n\ge 0$. If $(N_1,N_2)=split(N;1-c,c)$, $0<c<1$, then
the conditional tail distribution $P\{N_2>j|N_1=k\}$ is light-tailed as $j\to\infty$ for any fixed $k\ge 0$, or we simply say that $N_2|N_1=k$ has a lighted-tail distribution.
\end{lemma}
\begin{proof}
Since
\begin{equation}
P\{N_1=k,N_2>j\}= \sum_{n=j+1}^{\infty} p_{k+n} \binom{k+n}{k} (1-c)^k c^n
\le  \sum_{n=j+1}^{\infty} \binom{k+n}{k} c^n,
\end{equation}
there exists some $d\in (c,1)$ such that $P\{N_1=k,N_2>j\}= O(d^j)$ for large $j\ge 0$.
\end{proof}
\begin{remark}\label{remark-lighted}
By Remark \ref{M{a1},M{a2}}, we know $(M_{a1},M_{a2})\stackrel{\rm d}{=}split(N_{\lambda}(T_{\beta}^{(e)});q,p)$. So
$M_{a2}|M_{a1}=0$ has a light-tailed probability distribution.
\end{remark}

\subsubsection{Tail probability $P\{H_2>j|H_1=0\}$}

In this subsection, we prove a proposition which will be used for studying the tail asymptotic behaviour of $M_{b2}|M_{b1}=0$ later in Subsection~\ref{subsec-Mb2>j|Mb1=0}.
\begin{proposition}\label{lemma-Hbeta12>j}
As $j\to\infty$,
\begin{equation}
P\{H_2>j|H_1=0\}\sim \left[\frac {\beta(\lambda_1)} {1-\beta(\lambda_1)}\cdot \frac {\lambda_2^{a}} {(1-\rho_1)^{a+1}}\right] \cdot j^{-a}L(j).\label{lemma-Hbeta12>j-formula}
\end{equation}
\end{proposition}

To prove this proposition, we need the following two lemmas (Lemma~\ref{lem:gamma is LST} and Lemma~\ref{proposition-A1}).
Note that $H_2|H_1=0$ has the PGF $E(z_2^{H_2}|H_1=0)= H(0,z_2)/H(0,1)$.
Setting $z_1=0$ in (\ref{G2-decomp-H-x}) and noting $h(z_2)=\alpha(\lambda_2-\lambda_2 z_2)$, we obtain
\begin{eqnarray}\label{PGF-Hbeta12}
E(z_2^{H_2}|H_1=0)
=\frac 1 {1-\beta(\lambda_1)}\cdot \frac {\beta(\lambda-\lambda_1\alpha(\lambda_2-\lambda_2 z_2)-\lambda_2 z_2)-\beta(\lambda-\lambda_2 z_2)} {\alpha(\lambda_2-\lambda_2 z_2)}=\gamma(\lambda_2-\lambda_2 z_2),
\end{eqnarray}
where
\begin{eqnarray}
\gamma(s) = \frac 1 {1-\beta(\lambda_1)}\cdot \frac {\beta(s+\lambda_1-\lambda_1\alpha(s))-\beta(s+\lambda_1)} {\alpha(s)}.\label{Apdx-1}
\end{eqnarray}

\begin{lemma} \label{lem:gamma is LST}
$\gamma(s)$ is the LST of a probability distribution on $[0,\infty)$.
\end{lemma}

\proof
By Theorem~1 in Feller~\cite{Feller1971} (see p.439),
it is true iff $\gamma(0)=1$ and $\gamma(s)$ is completely monotone, i.e., $\gamma(s)$ possesses derivatives of all orders such that $(-1)^{n}\frac {d^n}{ds^n}\gamma(s)\ge 0$ for $s> 0$, $n\ge 0$. It is easy to check by (\ref{Apdx-1}) that $\gamma(0)=1$. Next, we verify that $\gamma(s)$ is completely monotone. The following criterion is often used to verify that a function is completely monotone.

\noindent{\bf Criterion A.1}\ (p.441 in \cite{Feller1971}) If $\vartheta_1(\cdot)$ and $\vartheta_2(\cdot)$ are completely monotone so is their product $\vartheta_1(\cdot)\vartheta_2(\cdot)$.

By using Taylor expansion, we have from (\ref{Apdx-1}) that
\begin{eqnarray}
\gamma(s) = \frac 1 {1-\beta(\lambda_1)}\cdot \sum_{k=1}^{\infty}\frac {\beta^{(k)}(s+\lambda_1)} {k!}(-1)^k\lambda_1^k(\alpha(s))^{k-1}. \label{gamma(s)-Taylor}
\end{eqnarray}

\noindent{\it Fact 1.} Since $\beta(s)$  is completely monotone, we know that $(-1)^k \beta^{(k)}(s+\lambda_1)$, $k\ge 1$, are completely monotone by the definition of completely monotone functions.

\noindent{\it Fact 2.} Since $\alpha(s)$  is completely monotone, we know that $(\alpha(s))^{k}$, $k\ge 1$, are completely monotone by using Criterion A.1.

By (\ref{gamma(s)-Taylor}), Facts 1 and 2, and applying Criterion A.1 and the definition of completely monotone functions, we know that $\gamma(s)$ is completely monotone. Therefore, it is the LST of a probability distribution.
\pend

\begin{remark}\label{gamma-compound-possion}
Let $T_{\gamma}$ be a r.v. whose probability distribution has the LST $\gamma(s)$.
Then the expression $E(z_2^{H_2}|H_1=0)=\gamma(\lambda_2-\lambda_2 z_2)$,  in (\ref{PGF-Hbeta12}), implies that $H_2|H_1=0\stackrel{\rm d}{=}N_{\lambda_2}(T_{\gamma})$.
\end{remark}

\begin{lemma}\label{proposition-A1}
As $t\to\infty$,
\begin{eqnarray}
P\{T_{\gamma}>t\}&\sim&\frac {\beta(\lambda_1)} {1-\beta(\lambda_1)}\cdot c_{\alpha} t^{-a}L(t),
\end{eqnarray}
where $c_{\alpha}=1/(1-\rho_1)^{n+1}$.
\end{lemma}

\proof
By (\ref{busy-eqn-alpha}), we can rewrite (\ref{Apdx-1}) as
\begin{eqnarray}
\gamma(s)=\frac 1 {1-\beta(\lambda_1)}\cdot \left[1 - \frac {\beta(s+\lambda_1)} {\alpha(s)}\right].\label{Apdx-1b}
\end{eqnarray}
In the following, we will divide the proof of Lemma~\ref{proposition-A1} into two parts, depending on whether $a>1$ is an integer or not.

\noindent{\it Case 1: Non-integer $a >1$.}
Suppose that $n<a<n+1$, $n\in\{1,2,\ldots\}$. Since $P\{T_{\alpha}>t\}\sim c_{\alpha} t^{-a}L(t)$, we have its moments $\alpha_{n}<\infty$ and $\alpha_{n+1}=\infty$.
Define $\alpha_{n}(s)$ in a manner similar to that in (\ref{phi1}). Therefore, $\alpha(s)=1+\sum_{k=1}^{n}\frac{\alpha_k}{k!}(-s)^k+(-1)^{n+1}\alpha_{n}(s)$.
By Lemma~\ref{Cohen}, we know that
\begin{equation}\label{alpha{m-1}(s)asym}
\alpha_{n}(s) \sim \big[\Gamma(a-n)\Gamma(n+1-a)/\Gamma(a)\big]\cdot c_{\alpha} s^{a}L(1/s),\quad s\downarrow 0.
\end{equation}
Noting that $\frac 1 {1-x}=1 +x +x^2+\cdots$ for $|x|<1$, we have
\begin{eqnarray}
\frac 1 {\alpha(s)} &=& 1+\sum_{k=1}^{n}v_k (-s)^k -(-1)^{n+1}\alpha_{n}(s)+O(s^{n+1}),
\end{eqnarray}
where $v_k,\ k=1,2,\ldots,n$ are constants.
Since $\beta(s+\lambda_1) =\beta(\lambda_1)+ \sum_{k=1}^{\infty}\frac{\beta^{(k)}(\lambda_1)}{k!} s^k$, we have, from (\ref{Apdx-1b}),
\begin{eqnarray}
\gamma(s)&=&1 +\sum_{k=1}^{n}u_k (-s)^k +(-1)^{n+1}\frac {\beta(\lambda_1)} {1-\beta(\lambda_1)}\alpha_{n}(s)+O(s^{n+1}),
\end{eqnarray}
where $u_k,\ k=1,2,\ldots,n$ are constants. Define $\gamma_{n}(s)$ in a manner similar to that in (\ref{phi1}). By Lemma~\ref{Cohen}, we know that
\begin{eqnarray}
\gamma_{n}(s)&=&\frac {\beta(\lambda_1)} {1-\beta(\lambda_1)}\alpha_{n}(s)+O(s^{n+1})\nonumber\\
&\sim& \big[\Gamma(a-n)\Gamma(n+1-a)/\Gamma(a)\big]\frac {\beta(\lambda_1)} {1-\beta(\lambda_1)}\cdot c_{\alpha} s^{a}L(1/s),\quad s\downarrow 0.
\end{eqnarray}
Then, by making use of Lemma \ref{Cohen}, we complete the proof of Lemma~\ref{proposition-A1} for non-integer $a >1$.
\\
\\
{\it Case 2: Integer $a >1$.}
Suppose that $a=n\in\{2,3,\ldots\}$.  Since $P\{T_{\alpha}>t\}\sim c_{\alpha} t^{-n}L(t)$, we know that $\alpha_{n-1}<\infty$, but, whether $\alpha_{n}$ is finite or not remains uncertain. This uncertainty is essentially determined by whether $\int_x^{\infty}t^{-1}L(t)dt$ is convergent or not. Define $\widehat{\alpha}_{n-1}(s)$ in a way similar to that in (\ref{phi2}). Then, $\alpha(s)=1+\sum_{k=1}^{n-1}\frac{\alpha_k}{k!}(-s)^k+(-s)^{n}\widehat{\alpha}_{n-1}(s)$, and by Lemma~\ref{Lemma 4.5new}, $\widehat{\alpha}_{n-1}(xs)-\widehat{\alpha}_{n-1}(s)\sim -(\log x) c_{\alpha} L(1/s)/(n-1)!$ as $s\downarrow 0$. Therefore,
\begin{eqnarray}
\frac 1 {\alpha(s)} &=& 1+\sum_{k=1}^{n}v_k (-s)^k -(-s)^{n}\widehat{\alpha}_{n-1}(s)+o(s^{n})
\end{eqnarray}
where $v_1,v_2,\cdots,v_n$ are constants.
Since $\beta(s+\lambda_1) =\beta(\lambda_1)+ \sum_{k=1}^{\infty}\frac{\beta^{(k)}(\lambda_1)}{k!} s^k$, we have, from (\ref{Apdx-1b}),
\begin{eqnarray}
\gamma(s)&=&1+\sum_{k=1}^{n}u_k (-s)^k +\frac {\beta(\lambda_1)} {1-\beta(\lambda_1)}\cdot (-s)^{n}
\widehat{\alpha}_{n-1}(s) +o(s^{n}),\label{case2gamma(s)}
\end{eqnarray}
where $u_1,u_2,\cdots,u_n$ are constants. Now, define $\widehat{\gamma}_{n-1}(s)$ in a way similar to that in (\ref{phi2}), then it follows from (\ref{case2gamma(s)}) that
\begin{eqnarray}
\widehat{\gamma}_{n-1}(s)= u_{n}+\frac {\beta(\lambda_1)} {1-\beta(\lambda_1)}
\cdot\widehat{\alpha}_{n-1}(s) +o(1).
\end{eqnarray}
Therefore,
\begin{equation}
\lim_{s\downarrow 0}\frac {\widehat{\gamma}_{n-1}(xs)-\widehat{\gamma}_{n-1}(s)} {L(1/s)/(n-1)!}
=\frac {\beta(\lambda_1)} {1-\beta(\lambda_1)} \lim_{s\downarrow 0}\frac {\widehat{\alpha}_{n-1}(xs)-\widehat{\alpha}_{n-1}(s)} { L(1/s)/(n-1)!} = -(\log x)\frac {\beta(\lambda_1)} {1-\beta(\lambda_1)} c_{\alpha}.
\end{equation}
Applying Lemma \ref{Lemma 4.5new}, we complete the proof of Lemma~\ref{proposition-A1} for integer $a=n\in\{2,3,\ldots\}$.
\pend

\bigskip
\noindent\underline{{\sc Proof} of Proposition~\ref{lemma-Hbeta12>j}:} It follows directly from Remark~\ref{gamma-compound-possion}, Lemma~\ref{proposition-A1} and Lemma~\ref{lemma-Asmussen-poisson}. \pend

\subsubsection{Tail probability $P\{M_{b2}>j|M_{b1}=0\}$}\label{subsec-Mb2>j|Mb1=0}

By Remark \ref{remark-M{b1}M{b2}}, $(M_{b1},M_{b2})$ can be decomposed to a geometric sum of r.v.s which have the same probability distribution as that for $(H_{1},H_{2})$.

Now we are ready to study $P\{M_{b2}>j|M_{b1}=0\}$. Let $h_0=P \{H_1=0 \}$. Recall (\ref{remark-M{b1}M{b2}-eqn1}).
Since $P \{\sum_{i=1}^{n}H_1^{(i)}=0 \}=P \{H_1^{(1)}=\cdots =H_1^{(n)}=0 \}= h_0^n$, we immediately have
\begin{eqnarray}
P\{M_{b1}=0,M_{b2}>j\}&=&\rho_1\sum_{n=1}^{\infty} (1-\rho_1)\rho_1^{n-1}
 h_0^n\cdot P\Big \{\sum_{i=1}^{n}H_2^{(i)}>j\big |\sum_{i=1}^{n}H_1^{(i)}=0\Big \}\nonumber\\
&\sim & \rho_1\sum_{n=1}^{\infty} (1-\rho_1)\rho_1^{n-1}
 h_0^n \cdot n P \{H_{2}>j|H_{1}=0 \},\label{P{M-b1=0b}}
\end{eqnarray}
where we used the fact (property of subexponential distributions) that $P \{\sum_{i=1}^{n}H_2^{(i)}>j\big |\sum_{i=1}^{n}H_1^{(i)}=0 \}\sim n P \{H_{2}>j|H_{1}=0 \}$ as $j\to\infty$. Furthermore, by (\ref{P{M-b1=0b}}) and (\ref{lemma-Hbeta12>j-formula}),
\begin{equation}\label{4.41}
P\{M_{b2}>j|M_{b1}=0\}\sim c_{M_b} \cdot j^{-a} L(j),\quad j\to\infty,
\end{equation}
where $c_{M_b}>0$ is a constant.
\subsubsection{Tail probability $P\{R_{orb}>j|I_{ser}=1,R_{que}=0\}$}
By Remark \ref{remark3-3}, (\ref{def-T-theta}) and using (\ref{P{T-kappa>t}}), we have
\begin{eqnarray}\label{4.42}
P\{M_c>j\} \sim P\{T_{\eta}>j/\lambda_2\} &=& \vartheta P\{T_{\kappa}>j/\lambda_2\}\nonumber\\
&\sim&\frac {\lambda_2^a} {(1-\rho)(a-1)(1-\rho_1)^{a}}\cdot j^{-a+1}L(j),\quad j\to\infty.
\end{eqnarray}
As pointed out earlier in (\ref{(R{11},R{12})decomp}), $(R_{11},R_{12})$ can be stochastically decomposed into the sum of four independent r.v.s, which leads to
\begin{equation}\label{4.43}
P\{R_{12}>j|R_{11}=0\} = P\{M_{a2} + M_{b2} + M_{c} + R_{0}>j|M_{a1}=0, M_{b1}=0\}.
\end{equation}
It follow from (\ref{4.42}), (\ref{P{R_0>j}sim}) and (\ref{4.41}) and Remark \ref{remark-lighted} that $P\{R_{0}>j\}=o(P\{M_{c}>j\})$, $P\{M_{b2}>j|M_{b1}=0\}=o(P\{M_{c}>j\})$ and $P\{M_{a2}>j|M_{a1}=0\}=o(P\{M_{c}>j\})$. Applying Lemma \ref{Lemma 3.4new} to (\ref{4.43}), we have
\begin{equation}
P\{R_{12}>j|R_{11}=0\}\sim P\{M_{c}>j\},\quad j\to\infty.
\end{equation}
Note that $P\{R_{orb}>j|I_{ser}=1,R_{que}=0\}\equiv P\{R_{12}>j|R_{11}=0\}$. By (\ref{4.42}),
\begin{equation}
P\{R_{orb}>j|I_{ser}=1,R_{que}=0\}\sim\frac {\lambda_2^a} {(1-\rho)(a-1)(1-\rho_1)^{a}}\cdot j^{-a+1}L(j),\quad j\to\infty.
\end{equation}

\section*{Acknowledgments}
This work was supported in part by the National Natural Science Foundation of China (Grant No. 71571002),
the Natural Science Foundation of the Anhui Higher Education Institutions of China (No. KJ2017A340),
the Research Project of Anhui Jianzhu University, and a Discovery Grant from the Natural Sciences and Engineering Research Council of Canada (NSERC).

\appendix

\section{Appendix}

\subsection{Definitions and useful results from the literature}

\begin{definition}[Bingham, Goldie and Teugels~\cite{Bingham:1989}]
\label{Definition 3.1}
A measurable function $U:(0,\infty)\to (0,\infty)$ is regularly varying at $\infty$ with index $\sigma\in(-\infty,\infty)$ (written $U\in \mathcal{R}_{\sigma}$) iff $\lim_{t\to\infty}U(xt)/U(t)=x^{\sigma}$ for all $x>0$. If $\sigma=0$ we call $U$ slowly varying, i.e., $\lim_{t\to\infty}U(xt)/U(t)=1$ for all $x>0$.
\end{definition}
\begin{definition}[Foss, Korshunov and Zachary~\cite{Foss2011}]
\label{Definition 3.2}
A distribution $F$ on $(0,\infty)$ belongs to the class of subexponential distribution (written $F\in \mathcal S$) if \space $\lim_{t\to\infty}\overline{F^{*2}}(t)/\overline{F}(t)=2$, where $\overline{F}=1-F$ and $F^{*2}$ denotes the second convolution of $F$.
\end{definition}

It is well known that for a distribution $F$ on $(0,\infty)$, if $\overline{F}$ is regularly varying with index $-\sigma$, $\sigma\ge 0$ or $\overline{F}\in \mathcal{R}_{-\sigma}$, then $F$ is subexponnetial or $F\in \mathcal S$
(see, e.g., Embrechts et al. \cite{Embrechts1997}).

\begin{lemma}[de Meyer and Teugels~\cite{Meyer-Teugels:1980}]
\label{Lemma 2.1}
Under Assumption A,
\begin{equation}\label{Meyer-Teugels-T{alpha}}
P\{T_{\alpha}>t\}\sim \frac 1 {(1-\rho_1)^{a+1}}\cdot t^{-a} L(t)\quad \mbox{as }t\to\infty.
\end{equation}
\end{lemma}
The result (\ref{Meyer-Teugels-T{alpha}}) is straightforward due to the main theorem in \cite{Meyer-Teugels:1980}.

\begin{lemma}[pp.580--581 in \cite{Embrechts1997}]
\label{Embrechts-compound-geo}
Let $N$ be a discrete non-negative integer-valued r.v. with $p_k=P\{N=k\}$ such that for some $\varepsilon>0$, $\sum_{k=0}^{\infty} p_k(1+\varepsilon)^k<\infty$, and let
$\{Y_k\}_{k=1}^{\infty}$ be a sequence of non-negative, i.i.d. r.v.s having a common subexponential distribution $F$.
Define $S_n=\sum_{k=1}^n Y_k$. Then $P\{S_N > t\} \sim E(N)\cdot (1-F(t))$ as $t\to \infty$.
\end{lemma}

Two special cases:

\noindent (1) if $p_k=(1-\sigma)\sigma^{k-1}$, $0<\sigma<1$, $k\ge 1$, then $P\{S_N > t\} \sim (1-F(t))/(1-\sigma)$ as $t\to \infty$;

\noindent (2) if $p_k=\frac {\sigma^k} {k!} e^{-\sigma}$, $\sigma>0$, $k\ge 0$, then $P\{S_N > t\} \sim \sigma\cdot (1-F(t))$ as $t\to \infty$.

\begin{lemma}[Proposition~3.1 in  \cite{Asmussen-Klupperlberg-Sigman:1999}]
\label{lemma-Asmussen-poisson} Let $N_{\lambda}(t)$ be a Poison process with rate $\lambda$ and let $T$ be a positive r.v. with distribution $F$, which is independent of $N_{\lambda}(t)$.
If $\bar{F}(t)=P\{T>t\}$ is heavier than $e^{-\sqrt{t}}$ as $t\to\infty$, then $P(N_{\lambda}(T)>j)\sim P\{T>j/\lambda\}$ as $j\to\infty$.
\end{lemma}

Lemma~\ref{lemma-Asmussen-poisson} holds for any distribution $F$ with a regularly varying tail because it is heavier than $e^{-\sqrt{t}}$ as $t\to\infty$.

\begin{lemma}[p.48 in \cite{Foss2011}]
\label{Lemma 3.4new}
Let $F$, $F_1$ and $F_2$ be distribution functions.
Suppose that $F\in\mathcal S$.
If $\bar{F}_i(t)/\bar{F}(t)\to c_i$ as $t\to\infty$ for some $c_i\ge 0, \; i=1,2$, then
$\overline{F_1*F}_2(t)/\bar{F}(t)\to c_1+c_2$ as $t\to\infty$,
where the symbol $\bar{F}\stackrel{\rm def}{=}1-F$ and  ``$F_1*F_2$" stands for the convolution of $F_1$ and $F_2$.
\end{lemma}

To prove Lemma \ref{theorem-T-tail} and Lemma \ref{proposition-A1}, we list some notations and results.
Let $F(x)$ be any distribution on $[0,\infty)$ with the LST $\phi(s)$.
We denote the $n$th moment of $F(x)$ by $\phi_n$, $n\ge 0$.
It is well known that $\phi_n<\infty$ iff
\begin{equation}
\phi(s)=\sum_{k=0}^{n}\frac{\phi_k}{k!}(-s)^k + o(s^n),\quad n\ge 0.
\end{equation}
Next, if $\phi_n<\infty$, we introduce the notation $\phi_n(s)$ and $\widehat{\phi}_n(s)$, defined by
\begin{eqnarray}
\phi_n(s)&\stackrel{\rm def}{=}&(-1)^{n+1}\left\{\phi(s)-\sum_{k=0}^{n}\frac{\phi_k}{k!}(-s)^k\right\},\quad n\ge 0,\label{phi1}\\
\widehat{\phi}_n(s)&\stackrel{\rm def}{=}&\phi_n(s)/s^{n+1},\quad n\ge 0.\label{phi2}
\end{eqnarray}

It follows that if $\phi_n<\infty$, then $\lim_{s\downarrow 0}\widehat{\phi}_{n-1}(s)=\phi_n/n!$ and
$s\widehat \phi_{n}(s)=\phi_{n}/n!-\widehat \phi_{n-1}(s)$ for $n\ge 1$.

In addition, if $\phi_n<\infty$, we define a sequence of functions $F_k$ recursively by:
$F_1(t)=F(t)$ and $1-F_{k+1}(t) \stackrel{\rm def}{=} \int_t^{\infty}(1-F_k(x))dx$, $k=1,2,\ldots,n$.
It is not difficult to check that $1-F_{k+1}(0)=\phi_k/k!$ and $\int_0^t(1-F_k(x))dx$ has the LST $\widehat \phi_{k-1}(s)$.
Namely, for $k=1,2,\ldots,n$,
\begin{eqnarray}
\int_0^{\infty}e^{-st}(1-F_{k}(t))dt &=& \frac {\phi_{k-1}} {(k-1)!}\frac 1 s-\frac 1 s\int_0^{\infty}e^{-st}dF_{k}(t)\nonumber\\
&=&\frac {\phi_{k-1}} {(k-1)!}\frac 1 s-\frac 1 s\int_0^{\infty}e^{-st}(1-F_{k-1}(t))dt \nonumber\\
&=&\frac {\phi_{k-1}} {(k-1)!}\frac 1 s-\frac {\phi_{k-2}} {(k-2)!}\frac 1 {s^2}+\cdots+(-1)^{k-1} \frac 1 {s^k}+(-1)^{k}\frac 1 {s^k}\int_0^{\infty}e^{-st}dF_{1}(t) \nonumber\\
&=&\widehat \phi_{k-1}(s).\label{laplace-phi(t)}
\end{eqnarray}

\begin{lemma}[pp.333-334 in \cite{Bingham:1989}]
\label{Cohen}
Assume that $n<d<n+1$, $n\in\{0,1,2,\ldots\}$. Then, the following are equivalent:
\begin{eqnarray}
1-F(t) &\sim&  t^{-d} L(t),\quad t\to\infty,\\
\phi_n(s) &\sim& \big[\Gamma(d-n)\Gamma(n+1-d)/\Gamma(d)\big]\cdot s^{d}L(1/s),\quad s\downarrow 0.
\end{eqnarray}
\end{lemma}

\begin{definition}[e.g., Bingham et al. (1989) \cite{Bingham:1989}, p.128]
A function $F:(0,\infty)\to (0,\infty)$ belongs to the de Haan class $\Pi$ at $\infty$ if there exists a function $H:(0,\infty)\to (0,\infty)$ such that
$\lim_{t\uparrow \infty}\frac {F(xt)-F(t)}{H(t)} = \log x$ for all $x>0$,
where the function $H$ is called the auxiliary function of $F$.
\end{definition}

\begin{lemma}
\label{Lemma 4.5new}
Assume that $n\in\{1,2,\ldots\}$. Then, the following two statements are equivalent:
\begin{eqnarray}
&&1-F(t)\sim t^{-n}L(t),\quad t\to\infty; \label{deHaan3}\\
&&\lim_{s\downarrow 0}\frac {\widehat{\phi}_{n-1}(xs)-\widehat{\phi}_{n-1}(s)}{L(1/s)/(n-1)!}=-\log x, \quad\mbox{ for all }x>0.\label{deHaan4}
\end{eqnarray}
\end{lemma}

\proof
Recall the definition of $F_{k}(t)$. Repeatedly using Karamata's theorem (p.27 in \cite{Bingham:1989}) and the monotone density theorem (p.39 in \cite{Bingham:1989}),
we know that $1-F(t)\sim t^{-n}L(t)$ is equivalent to $1-F_n(t)\sim t^{-1}L(t)/(n-1)!$, which in turn is equivalent to
$\int_0^t(1-F_n(x))dx\in \Pi$ with an auxiliary function that can be taken as $L(t)/(n-1)!$ (see, e.g., p.335 in \cite{Bingham:1989}).
By (\ref{laplace-phi(t)}),  $\int_0^t(1-F_n(x))dx$ has the LST $\widehat{\phi}_{n-1}(s)$. Applying Theorem~3.9.1 in \cite{Bingham:1989} (pp.172--173), we complete the proof.
\pend

\end{document}